\documentclass{article}
\usepackage[utf8]{inputenc}
\usepackage[margin=1 in]{geometry}
\usepackage{amsmath}
\usepackage{amsfonts}
\usepackage{amssymb}
\usepackage{amsthm}
\usepackage{mathrsfs}
\usepackage{tikz-cd}
\usepackage{enumitem}
\usepackage{lmodern}
\usepackage{comment}
\usepackage{hyperref}
\usepackage{graphicx}
\usepackage{float}

\newtheorem{theorem}{Theorem}
\numberwithin{theorem}{section}
\newtheorem{lemma}{Lemma}
\numberwithin{lemma}{section}

\newtheorem{corollary}{Corollary}
\numberwithin{corollary}{section}

\newtheorem{question}{Question}
\numberwithin{question}{section}

\theoremstyle{definition}
\newtheorem{definition}{Definition}
\numberwithin{definition}{section}

\theoremstyle{remark}
\newtheorem{remark}{Remark}

\newcommand{\mb}{\mathbb}
\newcommand{\tb}{\textbf}

\newcommand{\R}{\mathbb{R}}

\newcommand{\se}{\subseteq}

\newcommand{\mc}{\mathcal}

\newcommand{\eps}{\varepsilon}
\renewcommand{\vec}{\overrightarrow}

\title{Balanced configurations of points in the plane}
\author{Laura Pierson \\ Harvard University \\ \href{mailto:lcpierson73@gmail.com}{lcpierson73@gmail.com} \and Julian Wellman \\ Massachusetts Institute of Technology \\ \href{mailto:wellman@mit.edu}{wellman@mit.edu}}
\date{}

\begin{document}

\maketitle

\begin{abstract}
A \emph{balanced configuration} of points on the sphere $S^2$ is a (finite) set of points which are in equilibrium if they act on each other according any force law dependent only on the distance between two points. The configuration is additionally \emph{group-balanced} if for each point in a configuration $\mc{C}$, there is a symmetry of $\mc{C}$ fixing only that point and its antipode. Leech showed that these definitions are equivalent on the sphere $S^2$ by classifying all possible balanced configurations. On the other hand, Cohn, Elkies, Kumar, and Sch\"urmann showed that for $n\ge 7,$ there are examples of balanced configurations in $S^{n-1}$ which are not group balanced. They also suggested extending the notion of balanced configurations to Euclidean space, and conjectured that at least in the case of the plane, all discrete balanced configurations in $\R^n$ are group-balanced. We verify a reformulation of this conjecture by providing a complete classification of the balanced configurations in $\R^2$ satisfying a certain minimal distance property.
\end{abstract}

\section{Introduction}

\subsection{Balanced configurations on the sphere}

Consider a set of points acting on each other according to some force law, so that each point exerts a force on each other point. A natural question that arises in physics is what the equilibrium configurations are under a certain force law, i.e. how the points can be arranged so that there is no net force on any point. In a physical context, this question is related to energy minimizing configurations, since any configuration which minimizes the total energy must be in equilibrium. The question of arranging points on a sphere to minimize the total energy under a usual inverse square force law was first posed by Thomson \cite{Thomson} in the context of his model of the atom. Abstracting this, Cohn and Kumar \cite{CK} asked in what ways a set of points can be arranged to simultaneously minimize the total energy under \emph{every} force law where the force exerted by one point on another depends only on the distance between them and decreases with distance, and Ballinger, Blekherman, Cohn, Giansiracusa, Kelly, and Sch\"urmann \cite{BBCGKS} studied this question experimentally.

One can also ask more generally which configurations of points will be in equilibrium under an arbitrary force law dependent only on distance, without regard to whether the total energy is minimized. This question was first posed by Leech \cite{Leech} in the case of points on a sphere. He defined the following notion of \emph{balanced configurations}:

\begin{definition}
[Leech~\cite{Leech}]
\label{def:bal_sphere}
Let $S^{n-1}$ denote the unit sphere in $\R^n$ and $O$ the origin. A finite configuration $\mc{C}\se S^{n-1}$ is \emph{balanced} if for any $P\in\mc{C}$ and $d\ge 0$, the points $P_1,\dots,P_m\in \mc{C}$ at distance $d$ from $P$ have a vector sum $\vec{OP_1}+\dots+\vec{OP_m}$ which is a scalar multiple of $\vec{OP}.$
\end{definition}

Leech then classified all finite balanced subsets of the sphere. His classification can be stated as follows:

\begin{theorem}
[Leech~\cite{Leech}]
\label{thm:leech_classification}
Each tiling of the sphere with regular polygons can be turned into a balanced configuration by taking any one, two or all three of the following sets:
\begin{itemize}
    \item The vertices of the tiles,
    \item The midpoints of the edges of the tiles,
    \item The centers of the faces of the tiles.
\end{itemize}
Moreover, any nonempty, finite balanced configuration of points on the sphere arises in this manner.
\end{theorem}

In the case of the sphere, there is a tiling with regular polygons corresponding to each of the five Platonic solids, as well as for each $n$ a tiling splitting the sphere into two regular $n$-gons meeting along some great circle. Thus, each balanced configuration of the sphere is either some combination of the vertices, edge midpoints, and face centers of a regular polyhedron, or a set of $n$ evenly spaced points along a great circle, possibly with the two antipodal points added.

In \cite{cohn2010}, Cohn, Elkies, Kumar, and Sch\"urman introduced the notion of \emph{group-balanced} subsets of a sphere:

\begin{definition}
[\cite{cohn2010}]
\label{def:gp_bal_sphere}
A configuration $\mc{C}\se S^{n-1}$ is \emph{group-balanced} if for any $P\in\mc{C}$, there is a symmetry of $\mc{C}$ fixing no points of $S^{n-1}$ except $P$ and its antipode.
\end{definition}

Any group-balanced configuration must also be balanced. To see this, suppose $g$ is an isometry of the sphere which is a symmetry of $\mc{C}$ fixing no points except $P$. Then if $P_1,\dots,P_m$ are the neighbors of $P$ at some distance $d$, then $g(P_1),\dots,g(P_m)$ is some permutation of $P_1,\dots,P_m$, so $\vec{OP_1}+\dots+\vec{OP_m}$ is fixed by $g$, and must therefore be a scalar multiple of $\vec{OP}.$ It follows from Leech's classification that any balanced subset $\mc{C}$ of $S^2$ is group-balanced, since there is a nontrivial rotational symmetry about every point. However, this is not always the case in higher dimensions, as shown in \cite{cohn2010} via constructions of point configurations in $S^{n-1}$ for $n\ge 7$ which are balanced but not group-balanced. The question of whether or not all balanced configurations on the sphere are group-balanced for $n=4,5,$ and 6 is still open.

\subsection{Balanced configurations in Euclidean space}

The authors of \cite{cohn2010} introduced the following notions of balanced and group-balanced configurations of points in Euclidean space, analogous to Definitions \ref{def:bal_sphere} and \ref{def:gp_bal_sphere}:

\begin{definition}
[\cite{cohn2010}]
A configuration $\mc{C}\se \R^n$ is \emph{balanced} if for any point $P\in\mc{C}$ and any distance $d$, the points $P_1,P_2,\dots,P_m\in\mc{C}$ at distance $d$ from $P$ satisfy $\vec{PP_1}+\dots+\vec{PP_m}=\vec{0}.$
\end{definition}

\begin{remark}\label{rem:alternate_def}
One could alternately state Definition \ref{def:bal_sphere} as saying that for any $P\in\mc{C}$ and $d\ge 0$, if $P_1,\dots,P_m$ are the points of $\mc{C}$ at distance $d$ from $P$, and $P_1',\dots,P_m'$ their projections onto the hyperplane tangent to the sphere at $P$, then $\vec{PP_1'}+\dots+\vec{PP_m'}=\vec{0}$. This highlights the analogy between the two definitions.
\end{remark}

\begin{definition}
[\cite{cohn2010}]
A configuration $\mc{C}\se\R^n$ is \emph{group-balanced} if for any $P\in\mc{C}$, there is a symmetry of $\mc{C}$ fixing no points of $\R^n$ except $P$.
\end{definition}

By a similar argument to the spherical case, all group-balanced configurations are balanced. The authors of \cite{cohn2010} conjectured that, like in the spherical case, all discrete balanced configurations in $\R^2$ are group-balanced \cite[Conjecture 5.1]{cohn2010}, but that for sufficiently large $n$ this is not always the case in $\R^n$ \cite[Conjecture 5.2]{cohn2010}. The condition that the configuration be \emph{discrete} is intended to correspond to the condition implicitly assumed by Leech~\cite{Leech} that the configuration is finite on the sphere. We will need to use a slightly stronger assumption, which we believe also includes all the configurations of interest.

\begin{definition}\label{def:mindist}
A configuration $\mc C$ has the \emph{minimal distance property} if there is some distance $d>0$ such that for all distinct $x,y\in \mc C$, the distance between $x$ and $y$ is at least $d$, and there is a pair of points $x,y\in \mc C$ which are separated by a distance of exactly $d$. 
\end{definition}


This property excludes configurations such as the entire plane or an entire line which would otherwise be balanced, as well as stranger configurations like the set of rational points in the plane. We believe that all discrete balanced configurations satisfy the minimal distance property, but do not have a proof. Another natural analogue of the finiteness condition on spheres could be that the configuration is periodic with a finite fundamental domain containing finitely many points. Any such periodic configuration certainly satisfies the minimal distance property, but we will just assume the weaker minimal distance property.


\begin{theorem}\label{thm:classification}
A tiling of the plane with regular hexagons or a tiling with congruent parallelograms whose vertices form a lattice can be turned into a balanced configuration of points in $\R^2$ by taking one, two, or all three of the following sets:
\begin{itemize}
    \item The vertices of the tiles,
    \item The midpoints of the edges of the tiles,
    \item The centers of the faces of the tiles.
\end{itemize} Moreover, all balanced configurations of points in the plane that satisfy the minimal distance property and are not contained within a single line arise in this manner.
\end{theorem}

These configurations must all fall under one of the six types shown in Figure \ref{fig:all_balanced_configs}. For instance, the configuration consisting of the vertices and face centers of a lattice parallelogram tiling might seem to be a separate case, but is actually a lattice and so falls under Type 2.

\begin{figure}[H]
    \centering
    \includegraphics[width=13cm]{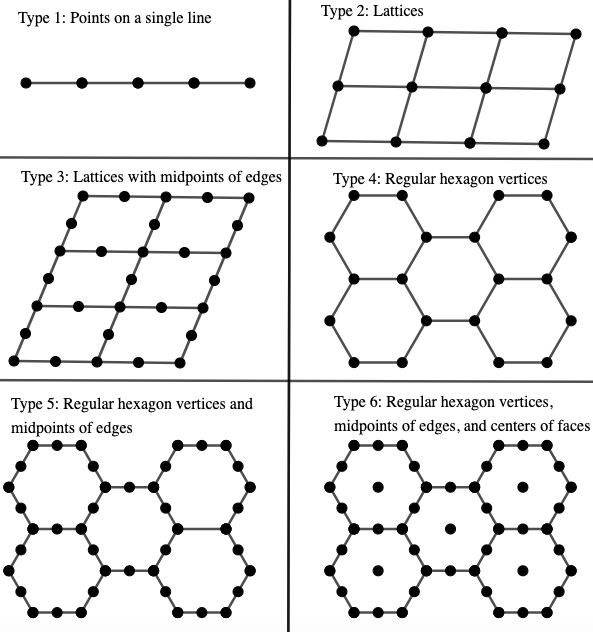}
    \caption{The six types of discrete balanced configurations of points in the plane.}
    \label{fig:all_balanced_configs}
\end{figure}

In each of these configurations there is a nontrivial rotational symmetry about every point, so they are all group-balanced as well. This corollary verifies a slightly weaker version of \cite[Conjecture 5.1]{cohn2010}:

\begin{corollary}
All balanced configurations in $\R^2$ satisfying the minimal distance property are group-balanced.
\end{corollary}

We will prove our main result, Theorem \ref{thm:classification} in Section \ref{sec:main_proof}. Then in Section \ref{sec:hyperbolic} we discuss the possibility of generalizing to the case of balanced configurations in the hyperbolic plane.

\section{Proof of the classification theorem}

\subsection{Setup and preliminary lemmas}\label{sec:main_proof}


\noindent\tb{Notation.} The balanced configurations we consider have the minimal distance property, so we can rescale them so that the minimal distance is exactly 1. Let $\mc{C}$ be a balanced configuration with minimal distance 1, and call two points in $\mc{C}$ \emph{neighbors} if they are at distance 1 from each other, and \emph{neighbors at distance $d$} if they are distance $d$ apart. Given two neighboring points $P$ and $Q$, we will consider the case where $P$ has $m$ neighbors and $Q$ has $n$ neighbors, for each pair $(m, n)$. We will refer to this case as $(m,n)$. Following the notation of Leech, we will let the neighbors of $P$ be $P_1,\dots,P_m=Q$ in counterclockwise order, and the neighbors of $Q$ be $Q_1,\dots,Q_n=P$ in counterclockwise order. We will also assume that $PQ$ is horizontal, with $P$ to the left of $Q$.\\

Following the approach used by Leech~\cite{Leech}, we will consider every case $(m,n)$ and determine what balanced configuration(s) can arise from each case, if any. It suffices to only consider the cases where $m \geq n$, and the following lemma implies that $m$ and $n$ must both be at most 6, so we have only finitely many cases to handle.

\begin{lemma}\label{lem:leq6}
Every point has at most 6 neighbors.
\end{lemma}

\begin{proof}
Suppose $P$ has $m$ neighbors $P_1,\dots,P_m$ at distance 1. Then $\angle P_iPP_{i+1} \geq 60^\circ$ for every $i$, because otherwise $P_iP_{i+1} < 1$. Thus, we get a sequence of $m$ angles whose sum is $360^\circ$ and each is at least $60^\circ$. Therefore $m\leq 6$.
\end{proof}

For the cases where $m=5$, we will use the following lemma, which is analogous to one Leech uses:

\begin{lemma}[cf. Leech~{\cite[Fact 2.9]{Leech}}] \label{lem:60-90}
If a point has exactly 5 neighbors, then the angles between adjacent neighbors are at least $60^\circ$ and strictly less than $90^\circ.$
\end{lemma}

\begin{proof}
Suppose $P$ has 5 neighbors $P_1,P_2,P_3,P_4,$ and $P_5.$ The fact that the angles are at least $60^\circ$ is immediate, since otherwise one of the distances between consecutive neighbors would be less than 1. Now suppose for contradiction $\angle P_1PP_5 \ge 90^\circ,$ and consider the components of the vectors $\vec{PP_1},\dots,\vec{PP_5}$ in the direction of the angle bisector of $\angle P_1PP_5$, shown to be horizontal in Figure \ref{fig:60-90_lemma}.

\begin{figure}[H]
    \centering
    \includegraphics[width=4cm]{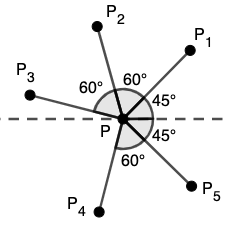}
    \caption{Showing that if $P$ has 5 neighbors, the angles between consecutive neighbors are less than $90^\circ.$}
    \label{fig:60-90_lemma}
\end{figure}

If we try to place $P_2,P_3,$ and $P_4$ as far to the right as possible, as shown in Figure \ref{fig:60-90_lemma}, then the sum of the horizontal components is $$2\cos 45^\circ + 2\cos 105^\circ + \cos 165^\circ \approx = -0.07 < 0.$$ Moving any of these points to the left will only make this sum smaller, so it can never be 0, and thus the configuration cannot be balanced if $\angle P_1PP_5\ge 90^\circ.$
\end{proof}

We will now proceed through the cases in reverse lexicographic order. A summary of all the cases we will consider and the corresponding balanced configurations is shown in Table \ref{tab:cases_list}.

\begin{table}[H]
    \centering
    \begin{tabular}{c|c}
        \tb{Case} & \tb{Balanced configurations} \\
        \hline
        $(6,6)$ & Vertices from an edge-to-edge equilateral triangle tiling \\
        \hline
        $(6,n)$ for $n<6$ & None \\
        \hline
        $(5,n)$ for $n\le 5$ & None \\
        \hline
        $(4,4)$ & Points of a lattice \\
        \hline
        $(4,3)$ & None \\
        \hline
        $(4,2)$ & Vertices and edge midpoints from a lattice parallelogram tiling \\
        \hline
        $(3,3)$ & Vertices from a regular hexagon tiling \\
        \hline
        $(3,2)$ & Vertices and edge midpoints from a regular hexagon tiling \\
        & Vertices, edge midpoints, and face centers from a regular hexagon tiling \\
        \hline
        $(2,2)$ & A single line of evenly spaced points \\
        & Points of a lattice \\
        & Vertices and edge midpoints from a lattice parallelogram tiling \\
    \end{tabular}
    \caption{Table of cases and the corresponding balanced configurations.}
    \label{tab:cases_list}
\end{table}





\subsection{Case \texorpdfstring{$(6,6)$}{(6,6)}}

We will show that the case $(6,n)$ is impossible for all $n<6$, so all the neighbors of $P$ and $Q$ must also have 6 neighbors. Similarly these neighbors must have 6 neighbors, and so on. Since the position of the first neighbor determines the positions of the other 5 exactly for each vertex, all points in the configuration connected to $P$ through a sequence of neighbors have their positions determined, if the case is possible at all. The equilateral triangle tiling achieves this case, so once $P$ and $Q$ are chosen we know they must be adjacent vertices in an equilateral triangle tiling. There is no room for additional vertices that do not have any neighbors in the tiling, so this is the only possible balanced configuration for the $(6,6)$ case.

\subsection{Case \texorpdfstring{$(6,n)$}{(6,n)} for \texorpdfstring{$n<6$}{n<6}}

In the case $n=6$ in the proof of Lemma~\ref{lem:leq6}, the neighbors of $P$ must be at perfect $60^\circ$ angles from each other, as shown below, since all 6 angles are at least $60^\circ$ and add up to only $360^\circ$.
\begin{figure}[H]
    \centering
    \includegraphics[width=8cm]{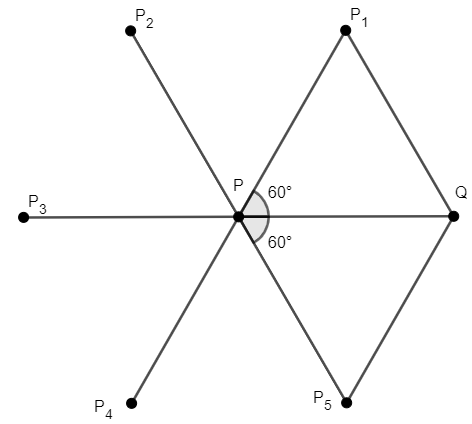}
    \caption{Showing that the case $(6,n)$ is impossible for $n<6$.}
\end{figure}
With the positions of $P$ and $Q$ fixed, this determines the positions of $P_1,\ldots,P_5$ exactly. Then $P_1,P_5$ are neighbors of $Q$. The sum of the neighbors of $Q$ so far is $\vec{QP}+\vec{QP_1} + \vec{QP_2} = 2\vec{QP}$. But the only way the remaining neighbors of $Q$ can sum to $-2\vec{QP}$ is if $Q$ has 3 more neighbors on the other side, since if it had only two they would both have to equal $-\vec{QP}$ to sum to $-2\vec{QP}$. This forces $n=6$, so this case is impossible.

\subsection{Case \texorpdfstring{$(5,5)$}{(5,5)}} 

In this case, we must have $P_1=Q_4$, because by Lemma ~\ref{lem:60-90}, $\angle P_1PQ$ and $\angle PQQ_4$ are both less than $90^\circ$, and so if $P_1\ne Q_4$ then $P_1$ and $Q_4$ would be less than 1 unit apart. Similarly, $P_4=Q_1.$ This implies that triangles $\triangle P_1PQ$ and $\triangle P_4PQ$ are equilateral, as shown below, so $\vec{PP_1}+\vec{PQ}+\vec{PP_4}=2\vec{PQ}$.
\begin{figure}[H]
    \centering
    \includegraphics[width=8cm]{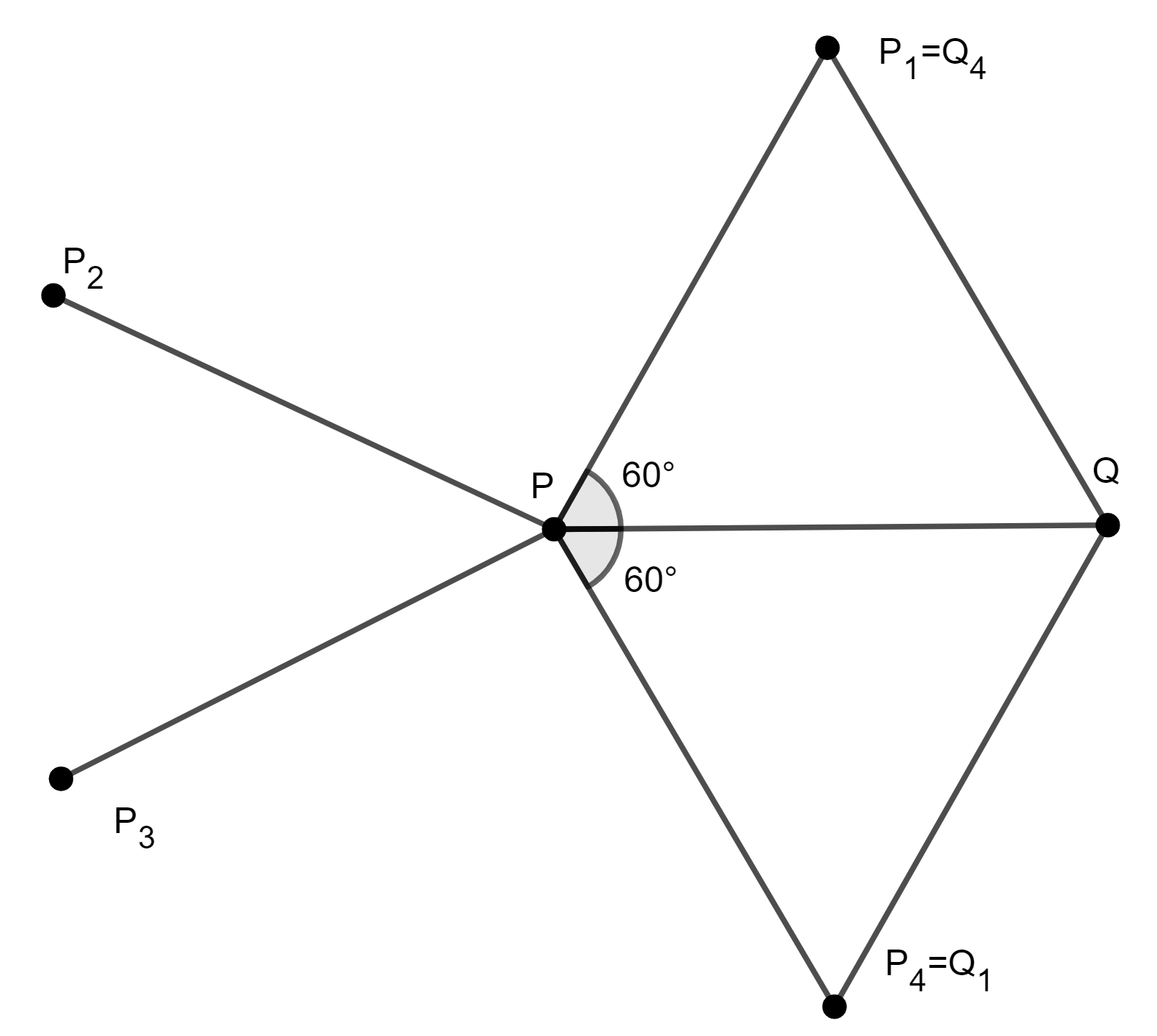}
    \caption{Showing that the case $(5,5)$ is impossible.}
\end{figure}
But there are only two remaining neighbors of $P$, which cannot sum to $-2\vec{PQ}$ because they cannot both be $-\vec{PQ}$. Thus, this case is impossible.

\subsection{Case \texorpdfstring{$(5,4)$}{(5,4)}}

The only way $Q$ can have exactly 4 neighbors is if $Q_3$ is the reflection of $Q_1$ about $Q$ and $Q_2$ is the reflection of $P$ about $Q$. This forces either $\angle PQQ_1< 90^\circ$ or $\angle PQQ_3< 90^\circ$, so assume $\angle PQQ_1< 90^\circ$. Then we must have $P_4=Q_1$, since otherwise $P_4$ and $Q_1$ would be too close together. Thus, $\triangle PQQ_1$ must be equilateral, so $\angle PQQ_1=60^\circ$ and $\angle PQQ_3=120^\circ$. Also, $\angle QQ_3Q_2 = 60^\circ$, so $Q_3$ has at least 4 neighbors (since if it had only 3 they would all have to be at angles of $120^\circ$).

Suppose for contradiction that $Q_3$ has 5 neighbors, and let $R$ be the next neighbor clockwise from $Q$, so $\angle QQ_3R < 90^\circ$. If we try to make $P_1$ and $R$ as far apart as possible, we would have $$\angle QPP_1 = \angle QQ_3R=90^\circ \implies \angle PQP_1=\angle Q_3QR = 45^\circ.$$ This gives the picture shown in Figure \ref{fig:54}. We find that $\triangle P_1QR$ is an isosceles triangle with $P_1Q=RQ=\sqrt{2}$ and $\angle P_1QR = 30^\circ,$ so the base angles are $75^\circ$ and thus $$P_1R = 2\sqrt{2}\sin 15^\circ \approx 0.73 < 1.$$ Thus, $Q_3$ cannot have 5 neighbors. But then $Q_3$ has exactly 4 neighbors, and so does $Q$, so this reduces to the case $(4,4)$, and we will later show that in that case, no point can have 5 neighbors. Thus, the case $(5,4)$ is impossible.

\begin{figure}[H]
    \centering
    \includegraphics[width=9cm]{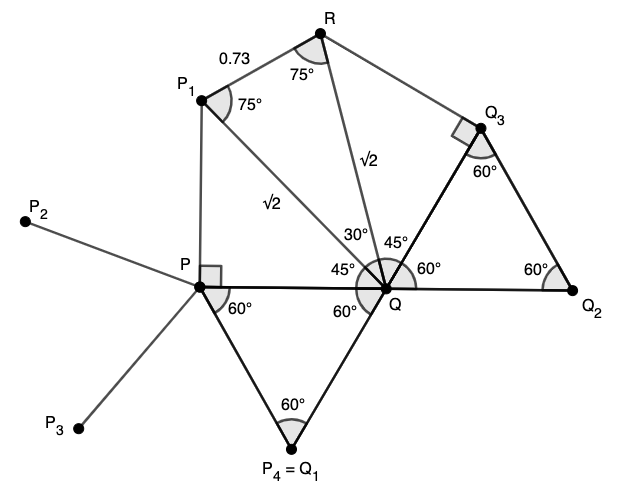}
    \caption{Case $(5,4)$ when $R$ and $P_1$ are as far apart as possible ($RP_1 \approx 0.73$).}
    \label{fig:54}
\end{figure}

\subsection{Case \texorpdfstring{$(5,3)$}{(5,3)}}

Since $Q$ has exactly 3 neighbors, we must have $\angle PQQ_1 = \angle Q_1QQ_2 = \angle Q_2QP = 120^\circ.$ If $P_1$ has 2, 4, or 5 neighbors, then this would reduce to one of the cases $(5,2)$, $(5,4)$, or $(5,5)$, all of which we will show to be impossible. Thus, we can assume $P_1$ has exactly 3 neighbors. Let $R$ be the next neighbor of $P_1$ counterclockwise from $P$. We claim that $RQ_2<1$. To make $R$ and $Q_2$ as far apart as possible, we would set $\angle P_1PQ = 90^\circ$ (since by Lemma~\ref{lem:60-90} it is at most $90^\circ$). Then we get the picture shown in Figure \ref{fig:53}.

\begin{figure}[H]
    \centering
    \includegraphics[width=7cm]{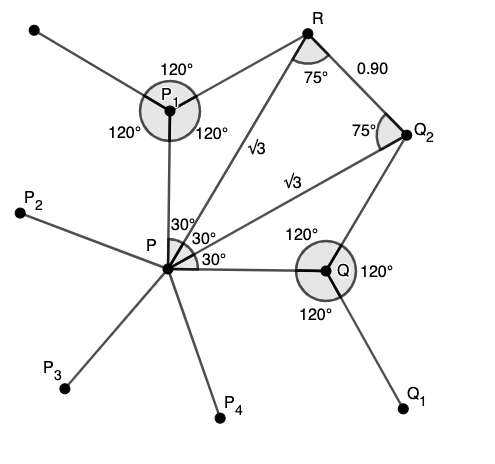}
    \caption{Case $(5,3)$ when $R$ and $Q_2$ are as far apart as possible ($RQ_2 \approx 0.90$).}
    \label{fig:53}
\end{figure}

Since $\angle QPQ_2 = \angle P_1PR = 30^\circ$ and $\angle P_1PQ=90^\circ$, we get $\angle RPQ_1 = 30^\circ$. Then $\triangle PQ_2R$ is isosceles with $PQ_2=PR=\sqrt{3}$ and $\angle PQ_2R = \angle PRQ_2=75^\circ,$ so $$Q_2R = 2\sqrt{3}\sin 15^\circ \approx 0.90 < 1.$$ This gives a contradiction, so the case $(5,3)$ is impossible.

\subsection{Case \texorpdfstring{$(5,2)$}{(5,2)}}

Since $Q$ has only two neighbors, we must have $QP_1,QP_4>1$. Now consider the neighbors of $Q$ at distance $QP_1$; we claim that $Q$ must have exactly 2 or 4 neighbors at this distance. 

If $Q$ had 3 neighbors at distance $QP_1$, then the next neighbor $R$ clockwise would be too close to $Q_1$, since the left part of Figure \ref{fig:52_1} shows the extreme case where $\angle P_1PQ = 90^\circ,$ and in that case $RQ_1\approx 0.52.$
\begin{figure}[H]
    \centering
    \includegraphics[width=12cm]{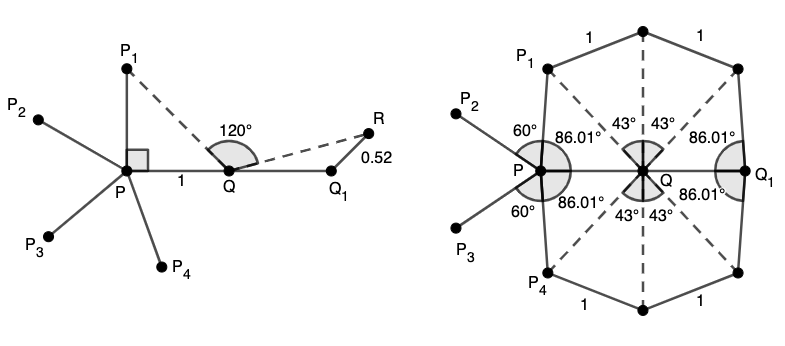}
    \caption{Showing that $Q$ cannot have 3, 5, or 6 neighbors at distance $QP_1$ in the case $(5,2).$}
    \label{fig:52_1}
\end{figure}
Next, suppose $Q$ has 5 neighbors at distance $QP_1.$ Since $\angle P_1PQ, \angle P_4PQ < 90^\circ$ by Lemma \ref{lem:60-90}, there is just barely space for 3 neighbors at this distance on either side of the line, so there cannot be more than 6 neighbors at this distance. There also cannot be exactly 5, because then there would be 3 points on one side and 2 on the other, and the sum of the vertical components on the side with 3 points would be larger even if we make the vertical components as small as possible on that side and as large as possible on the other. Finally, if there were exactly 6 such neighbors, then if we try to make $\angle P_1PQ$ and $\angle QPP_4$ as small as possible subject to the conditions that all the neighbors are at distance at least 1 from each other and from $Q$, we get $\angle P_1PQ = \angle QPP_4 \approx 86.01^\circ$ as shown in Figure \ref{fig:52_1}. But then the configuration will not be balanced around $P$ at distance 1, since even if we make $\vec{PP_2}$ and $\vec{PP_3}$ as close as possible to vertical (i.e. at $60^\circ$ angles to $\vec{PP_1}$ and $\vec{PP_4}$), the sum of the horizontal components of the vectors at that distance 1 from $P$ would be $$2\cos 146.01^\circ + 2\cos 86.01^\circ+1\approx-0.52<0,$$ and moving any of $P_1,P_2,P_3,$ and $P_4$ further left would only make that sum more negative, meaning there is no way for the sum to be 0.

Thus, $Q$ has exactly 2 or 4 neighbors at distance $QP_1$ and also at distance $QP_4$ by the same argument, so the reflections $P_1'$ and $P_4'$ of $P_1$ and $P_4$ about $Q$ must be in $\mc{C}$. Then $P_1'$ and $P_4'$ are both neighbors of $Q_1$ by symmetry, which forces $Q_1$ to have exactly 5 neighbors, since $\angle P_1'QP_4'<180^\circ$ by Lemma \ref{lem:60-90}, implying that it cannot have 2, 3, or 4 neighbors, and if it had 6 neighbors we would be in the case $(6,2).$

Next, we claim that $P_1,P_1',P_4,P_4'$ can each only have 2 neighbors at distance 1. In order to make the components in the direction of the angle bisector of $\angle P_1PP_4$ cancel while also having $\angle P_1P_2, \angle P_4PP_1 \ge 60^\circ,$ we must have $\angle P_1PP_4 \le 154^\circ$, since then these components would sum to at most $$2\cos 137^\circ +2\cos 77^\circ + 1\approx -0.01 < 0,$$ and making $\angle P_1PP_4$ larger can only make that sum more negative.

Now consider the balance about $P_1$ at distance $P_1P_4'.$ There must be some neighbor $R$ at this distance which is on or below line $P_1P_4'$, but there is no room for such a point without it being too close to one of $Q,P,P_4,P_3$, or $P_2.$ To see this, Figure \ref{fig:52_2} shows the extreme case where $R$ is as far as possible from $P_2$ (while also not being too close to any of the other points), which we obtain by setting $\angle P_1PQ=90^\circ$ and $\angle P_2PQ=\angle P_1PP_4=154^\circ,$ and $\angle P_4'P_1R=180^\circ.$ In that case we get $RP_2\approx 0.94 < 1,$ as shown, which implies that the case $(5,2)$ is impossible.

\begin{figure}[H]
    \centering
    \includegraphics[width=8cm]{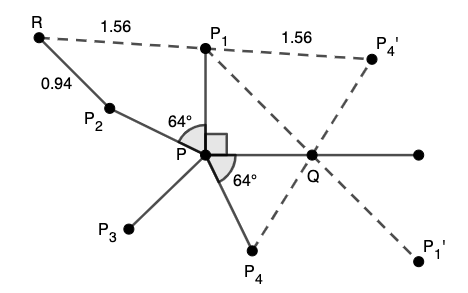}
    \caption{Showing that $P_1$ cannot be balanced at distance $P_1P_4'$ in the case $(5,2).$}
    \label{fig:52_2}
\end{figure}

\subsection{Case \texorpdfstring{$(4,4)$}{(4,4)}} 

In this case, $P_1,P,$ and $P_3$ are collinear, and so are $Q_1, Q$, and $Q_3.$ Then we must have $P_1P_3 \parallel Q_1Q_3,$ since otherwise we would get either $P_1Q_3<1$ or $P_3Q_1<1.$ This implies that $P_1Q_3 = P_3Q_1 = PQ = 1$, so $P_1$ and $Q_3$ are neighbors, as well as $P_3$ and $Q_1$. 

If angle $\alpha = 60^\circ$, then $P_1Q = 1$, which would put us in case $(6,6)$ instead. Then angles $\angle PP_1Q_3$ and $\angle QQ_3P_1$ are both not $120^\circ$, so $P_1$ and $Q_3$ must have more than three neighbors. Supposing angle $\alpha$ is acute, angle $\angle PP_1Q_3$ is obtuse (or otherwise $QQ_3P_1$ is obtuse and we use $Q_3$). By Lemma~\ref{lem:60-90}, this means that $P_1$ cannot have five neighbors. It cannot have six neighbors either, so it must have exactly four neighbors. Similarly, $Q_1$ has four neighbors, allowing us to extend our lattice-like shape to new vertices $A$ and $B$. 

\begin{figure}[H]\label{fig:44}
    \centering
    \includegraphics[width=13cm]{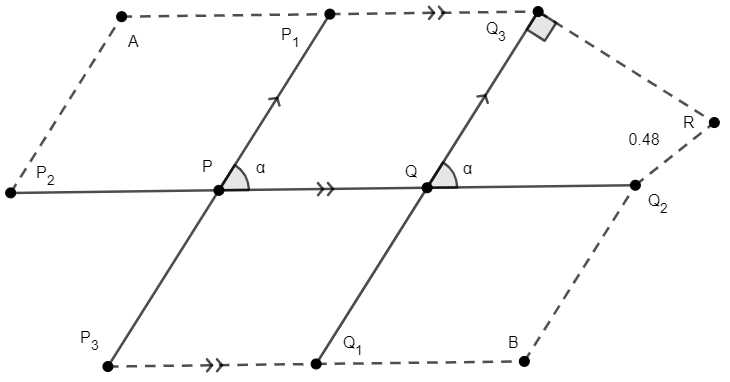}
    \caption{Case $(4,4)$, with extensions from $P_1,Q_1$ shown, and with hypothetical problematic edge from $Q_3$.}
\end{figure}

If $Q_3$ had five neighbors, then let $R$ be next one counterclockwise from $Q$, and we know from Lemma~\ref{lem:60-90} that $\angle QQ_3R < 90^{\circ}$. In Figure~\ref{fig:44} we see that the distance $Q_2R < 1$, so this is impossible and $Q_3$ must have only four neighbors. Then we can continue to extend our lattice indefinitely using this same reasoning at any pair of adjacent vertices that both are known to have four neighbors. Therefore all points of the form $P + a\vec{PQ} + b\vec{PP_1}$ for $a,b\in \mb Z$ are in the configuration. 

There is no room for any other points in the configuration without being within 1 unit of some other point. The only choice we had throughout this construction was in choosing the angle $\alpha$; any angle $60^\circ < \alpha < 90^\circ$ will work, resulting in tilings with any rhombus that is not too narrow.

\subsection{Case \texorpdfstring{$(4,3)$}{(4,3)}}

Since $P$ has exactly 4 neighbors, $P_1, P,$ and $P_3$ are collinear, so one of $\angle P_1PQ$ and $\angle QPP_3$ is acute. Assume $\angle P_1PQ < 90^\circ$, and consider the neighbors of $Q$ at distance $QP_1$.

First, we claim that it is impossible for $Q$ to have 5 or more neighbors at this distance. Suppose for contradiction it did, and let $R$ be the next neighbor of $Q$ clockwise from $P_1$ at distance $QP_1.$ We claim that $RQ_2<1.$ The extreme cases are when either $R$ is as far as possible to the left of $Q_2$ or as far as possible to the right of $Q_2.$ The first extreme would be when $P_1R = 1$ and $\angle P_1PQ = 90^\circ,$ and in this case we get $RQ_2 \approx 0.8 < 1$ (left half of Figure \ref{fig:43_1}).

The other extreme would be if we try to make $\angle P_1QR$ as large as possible (right half of Figure \ref{fig:43_1}). We claim that if $R$ is clockwise form $Q_1$, then $\angle Q_2QR\ge 120^\circ.$ To see this, if $\angle P_1QR < 120^\circ,$ then $\angle Q_2QR < \angle PQP_1,$ which implies $Q_2R < 1$, since if we compare $\triangle PP_1Q$ and $\triangle QQ_2R$, we have $QP=QQ_2$ and $QP_1=QR,$ so whichever of $\angle PQP_1$ or $\angle Q_2QR$ is larger will subtend the greater side, and $PP_1=1.$
\begin{figure}[H]\label{fig:43_1}
    \centering
    \includegraphics[width=13cm]{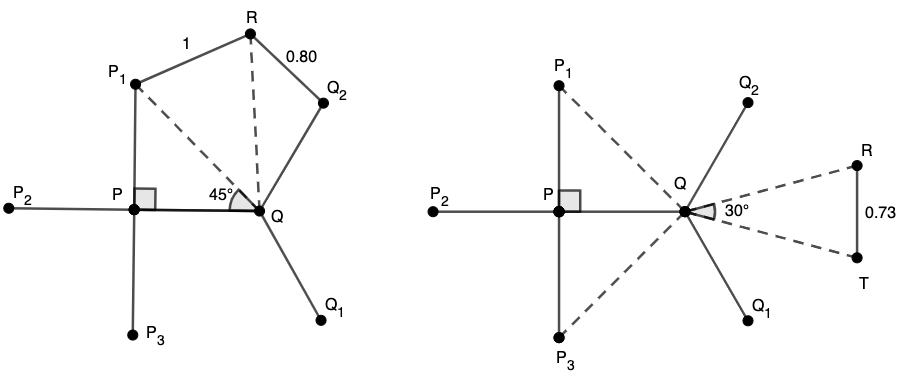}
    \caption{Showing that $Q$ cannot have 5 or more neighbors at distance $QP_1$ in the case $(4,3).$}
\end{figure}
Now let $S$ be the next neighbor of $Q$ counterclockwise from $P$ at distance $P_1Q$, followed by $T$. We must have $\angle PQS\ge \alpha,$ or else $PS < 1$, so $\angle P_1QS \ge 90^\circ.$ Then by the same reasoning as above, $\angle SQT > 120^\circ.$ But then $R$ and $T$ are too close together, since the greatest $RT$ can be is if $\angle P_1PQ = 45^\circ$ and $QP_1 = \sqrt{2}$, in which case $$RT = 2\sqrt{2}\sin 15^\circ \approx 0.73 < 1.$$ Thus, it is impossible for $Q$ to have 5 or more neighbors at distance $QP_1.$

Next, we claim that $Q$ cannot have exactly 2 or 4 neighbors at distance $QP_1$. If so, then the point $R$ which is the reflection of $P_1$ about $Q$ would have to be in $\mc{C}$, and $RQ_1<1$ because again $QR<\sqrt{2}$ and $\angle RQQ_1 < 15^\circ,$ as the most extreme case is when $\angle QPP_1 = 90^\circ,$ as shown in Figure \ref{fig:43_2}.
\begin{figure}[H]\label{fig:43_2}
    \centering
    \includegraphics[width=6cm]{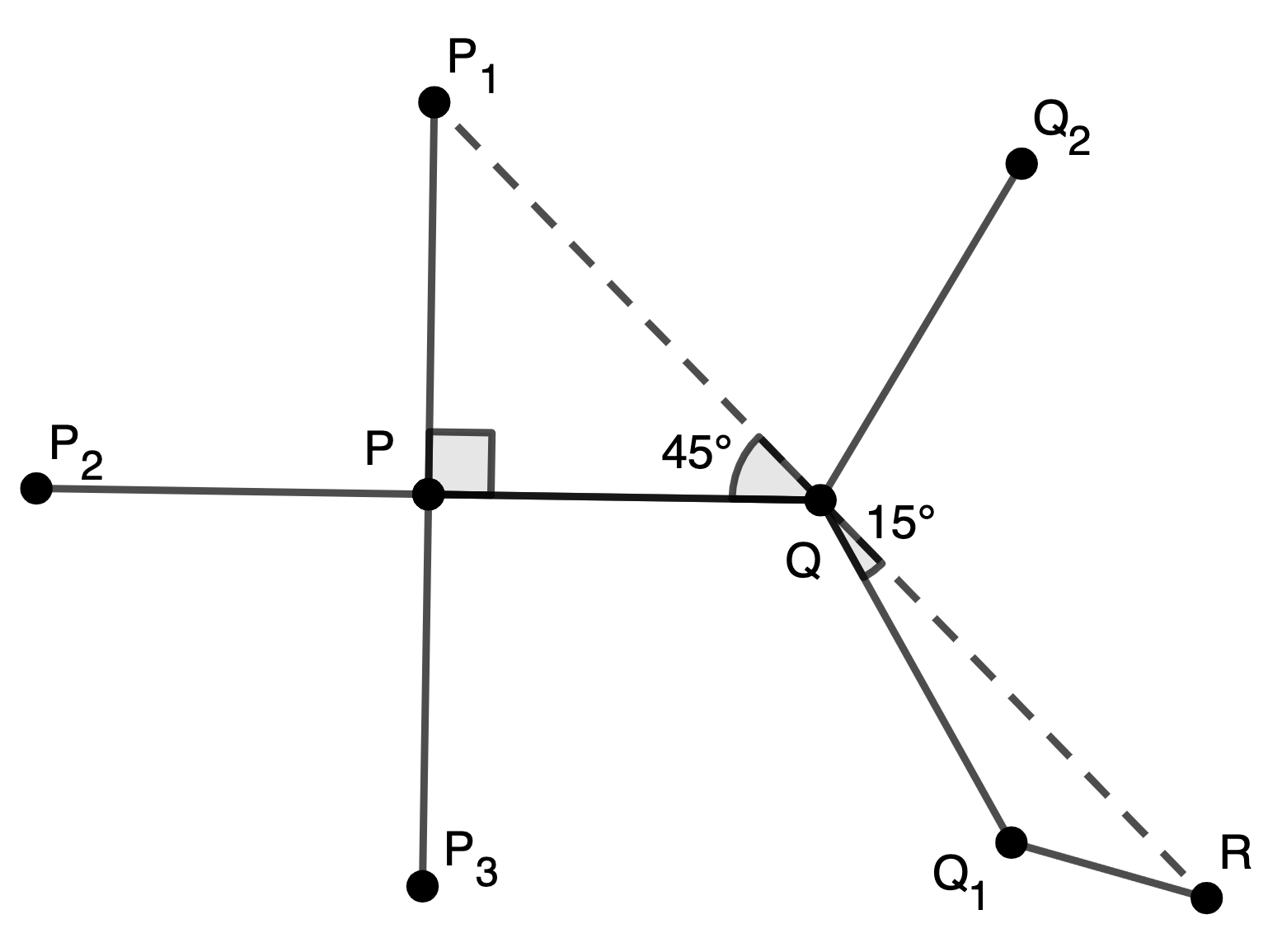}
    \caption{Showing that $Q$ cannot have 2 or 4 neighbors at distance $QP_1$ in the case $(4,3).$}
\end{figure}
The only remaining possibility is that $Q$ has 3 neighbors at distance $QP_1.$ Let $R$ be the next neighbor counterclockwise from $P_1$, so $\angle P_1QR = 120^\circ.$ We claim that $R$ is too close to $P_3,$ i.e. $P_3R<1.$ To show this, let $T$ be the point other than $Q$ on line $QR$ with $PT=1$, and let $\angle PP_1Q = \angle PQP_1 = \alpha$, as shown in Figure \ref{fig:43_3}.

\begin{figure}[H]\label{fig:43_3}
    \centering
    \includegraphics[width=8cm]{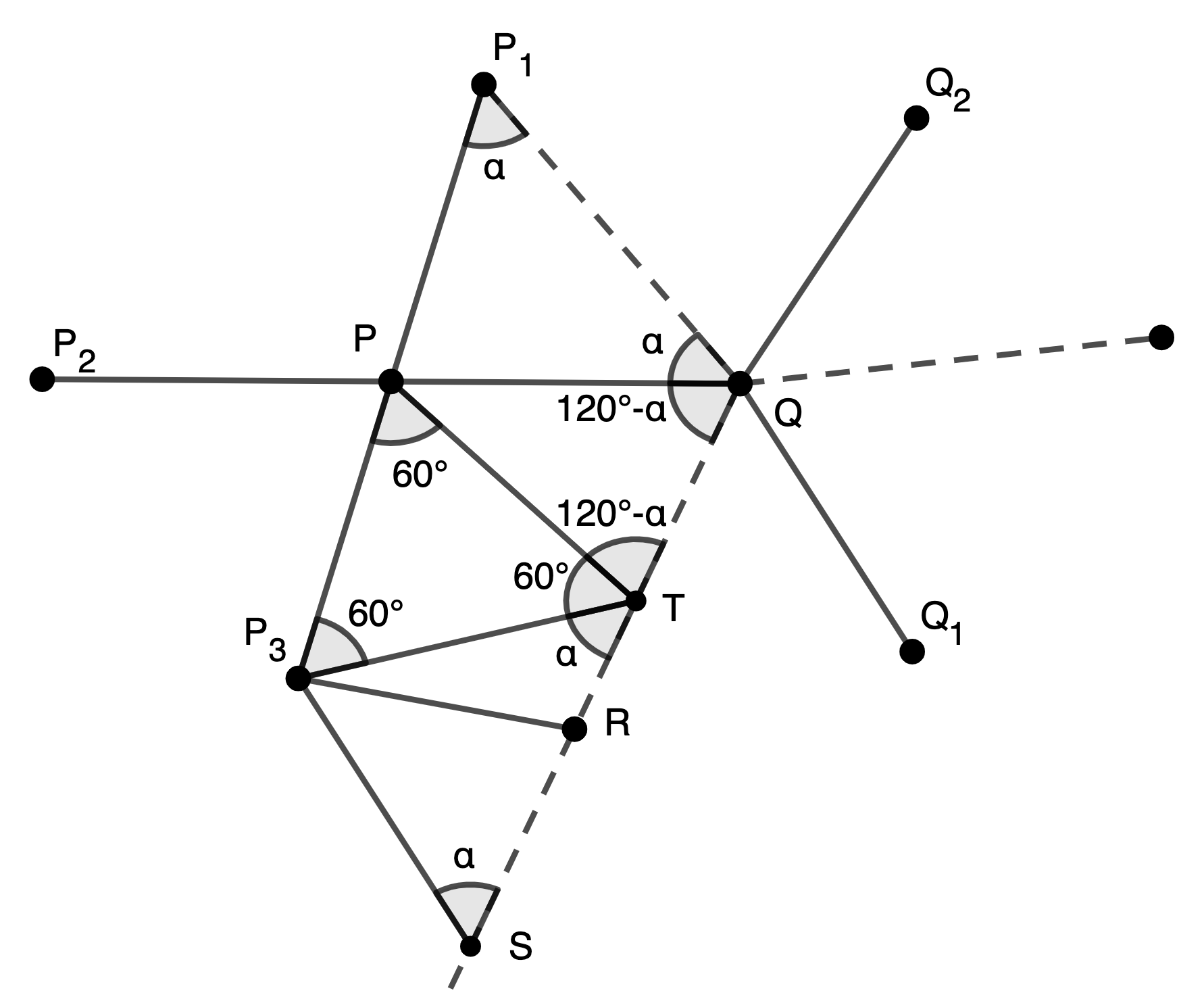}
    \caption{Case $(4,3)$ when $Q$ has 3 neighbors at distance $QP_1$}
\end{figure}

Then $\angle PQT = \angle PTQ = 120^\circ - \alpha$ since $\angle P_1QT = 120^\circ$ and $\triangle PQT$ is isosceles. But then for the angles of $PP_1QT$ to sum to $360^\circ$, we must have $\angle P_1PT = 120^\circ,$ so $\angle P_1PT = 60^\circ.$ Since $\triangle PP_3T$ is isosceles and has a $60^\circ$ angle, it must be equilateral, thus $P_3T = PP_3 = PT = 1.$ Let $S$ be the other point besides $T$ on line $QR$ with $P_3S = 1.$ Then since $\angle PTP_3 = 60^\circ$ and $\angle PTQ = 120^\circ - \alpha,$ we get $$\angle P_3ST = \angle P_3TS = \alpha = \triangle PP_1Q \cong \triangle P_3TS \implies P_1Q = ST.$$ Also, since $\angle P_1PQ > 60^\circ$ and $\angle TPP_1 = 120^\circ,$ $\angle PTQ < \angle QPP_1$ and thus $TQ < QP_1.$ Since $QR = QP_1,$ it follows that $R$ is between points $T$ and $Q$, which implies $P_3R < 1$, giving the desired contradiction.

This covers all possibilities for the case $(4,3)$, so this case is impossible.

\subsection{Case \texorpdfstring{$(4,2)$}{(4,2)}}

If $P$ has 4 neighbors, then $P_1$ and $P_3$ must be reflections of each other about $P$, and similarly $P_2$ and $Q$ are reflections of each other about $P$.

Now consider the neighbors of $Q$ at distance $QP_1$. We claim that it must have exactly 2 or 4 such neighbors. If $Q$ has exactly 3 neighbors at this distance, then the next neighbor $R$ clockwise from $P_1$ would be too close to $Q_1$, since in the most extreme case, $\angle PQP_1 = 45^\circ$, $P_1Q = QR =\sqrt{2}$, and $\angle RQQ_1 = 15^\circ$, which gives $Q_1R \approx 0.52< 1$, as shown in the left half of Figure \ref{fig:42_1}.

Now suppose $Q$ has at least 5 neighbors at distance $QP_1$ (second half of Figure \ref{fig:42_1}). First, we claim that it must have at most 6 such neighbors. We know that $\angle P_1QP \ge 45^\circ$. If $S$ and $T$ are the nearest to $Q_1$ of the neighbors of $Q$ at this distance, then we must also have $\angle SQQ_1, \angle TQQ_1 \ge \angle P_1QP \ge 45^\circ$ in order to have $SQ_1, Q_1 \ge 1.$ Then if $R$ is another neighbor between $P_1$ and $S$, we would need $\angle RQP_1, \angle RQS \ge 41.41^\circ$ (since this is the angle we get in the extreme case where $QR = QS = \sqrt{2}$. Thus, there is only space for one point $R$ between $P_1$ and $S$, so there can be at most 3 neighbors of $Q$ at distance $QP_1$ which are above line $PQ$, and by the same reasoning, there can also be at most 3 such neighbors below line $PQ$.

\begin{figure}[H]\label{fig:42_1}
    \centering
    \includegraphics[width=13cm]{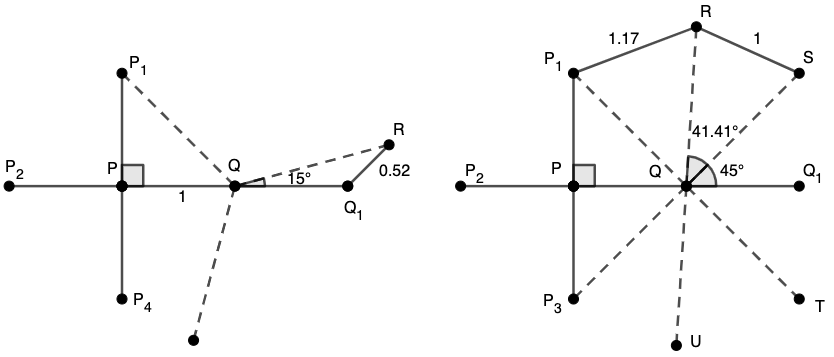}
    \caption{Showing that $Q$ must have 2 or 4 neighbors at distance $QP_1$ in the case $(4,2).$}
\end{figure}

Next, note that there is no way for $Q$ to have exactly 5 neighbors at distance $QP_1$, since it's impossible for the vertical components of the vectors to cancel: if we have 3 neighbors above the line and 2 below, say, then the vertical components of $\vec{QP}_1,$ $\vec{QR}$, and $\vec{QS}$ sum to more twice the length of $QP_1$, so there is no way the vertical components of the vectors below the line can cancel with them.

It follows that $Q$ must have exactly 6 neighbors at distance $QP_1$, assuming it has at least 5, but we claim that this is also impossible. The only way there is space for 3 such neighbors below line $PQ$ is if one of them is $P_3$, in which case we must have $QP_1 = QP_3$, so $\angle P_1PQ = \angle QPP_3 - 90^\circ.$ Now consider the balance of $P_1$ at distance $P_1R.$ First, note that we cannot have $P_1R = 1$: if so, we would get $90^\circ < \angle RP_1P < 120^\circ$, meaning $P_1$ cannot have 2, 3, 5, or 6 neighbors at distance 1 and thus must have exactly 4. But then the reflection of $P_1$ in $R$ would have to be one such neighbor, and this point would be less than 1 away from $P_2.$ If we try to make $P_1R$ as large as possible, we would take the extreme case where $RS=1$ and $\angle P_1QP = \angle SQQ_1 = 45^\circ$, which gives $P_1R \le 1.17.$ But then there is no space for another neighbor of $P_1$ at distance $P_1R$ which has smaller vertical coordinate than $P_1$, since any such point would be too close to at least one of $R, Q, P,$ or $P_2$, which means $P_1$ cannot be balanced at distance $P_1R$. This gives a contradiction and shows that $Q$ cannot have 5 or more neighbors at distance $QP_1.$

It follows that $Q$ has exactly 2 or 4 neighbors at distance $QP_1,$ so the reflection $P_1'$ of $P_1$ about $Q$ must be a point in $\mc{C}$. But then by symmetry, $P_1'Q=1$, so $P_1'$ is a neighbor of $Q_1$. Since $\angle QQ_1P_1'<90^\circ$, this rules out the possibility of $Q_1$ having 2 or 3 neighbors at distance 1, and we know from our work in the cases $(5,2)$ and $(6,2)$ that it cannot have 6 neighbors. Thus, $Q_1$ must have exactly 4 neighbors, so the reflection $P_3'$ of $P_3$ in $Q$ is also a neighbor (Figure \ref{fig:42_2}).
\begin{figure}[H]
    \centering
    \includegraphics[width=10cm]{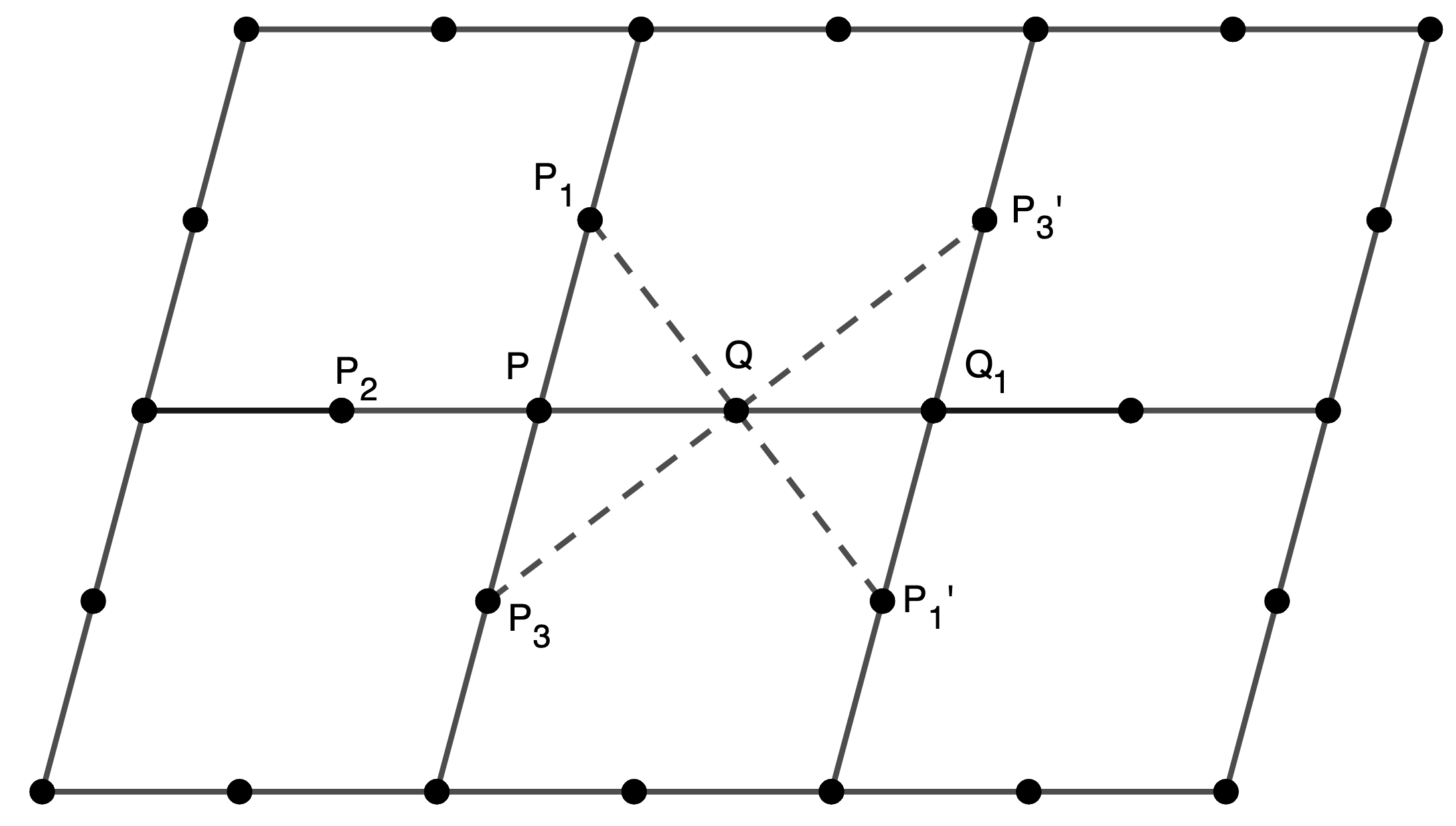}
    \caption{Resulting rhombus tiling for the case $(4,2).$}
    \label{fig:42_2}
\end{figure}
Now from our work in other cases, $(4,3)$ and $(5,4)$ are impossible, and $(4,4)$ implies that every point has exactly 4 neighbors, so the only way $(4,2)$ can happen is if all the neighbors of $P$ have exactly 2 neighbors. But then by the same argument as above, each of their neighbors has exactly 4 neighbors, and then each neighbor of one of those points has exactly 2 neighbors, and so on. We can continue this pattern forever, and we get a bipartite graph where the points alternate between having 2 or 4 neighbors. This gives the configuration shown in Figure \ref{fig:42_2}, which is a tiling of the plane with rhombuses with the midpoints of the edges added. 

This configuration is group balanced for any choice of $\angle P_1PQ$ between $60^\circ$ and $90^\circ,$ because there is a $180^\circ$ rotational symmetry about every point in $\mc{C}$, therefore it is balanced. Thus, the case $(4,2)$ works and gives a set of rhombus tilings where the points are the vertices of the tiles together with the midpoints of their edges.

\subsection{Case \texorpdfstring{$(3,3)$}{(3,3)}}

Consider the neighbors of $Q$ at distance $QP_1=QP_2=\sqrt{3}$. There are already two such neighbors to the left of $Q$ (namely, $P_1$ and $P_2$), so there must be at least two more to the right to balance them out. Any two consecutive neighbors must be separated by an angle of more than $30^\circ$ to be a distance at least 1 from each other, since if we have an isosceles triangle with a $30^\circ$ angle between two sides of length $\sqrt{3}$, the opposite side has length $2\sqrt{3}\cos 75^\circ \approx 0.90 < 1.$ This implies that there can be at most one more neighbor between $P_1$ and $Q_2$ (since $\angle P_1QQ_2 = 90^\circ$, at most two between $Q_2$ and $Q_1$ (since $\angle Q_2QQ_1 = 120^\circ$), and at most one between $Q_1$ and $P_2$ (since $\angle Q_1QP_2=90^\circ$), for a total of at most 6.

Next, we claim that $Q$ cannot have exactly 5 neighbors at this distance. If there were only one neighbor between $Q_2$ and $Q_1$, then there is no way the horizontal components can cancel, since $\vec{QP_1}+\vec{QP_2}$ already has horizontal component $-2<-\sqrt{3}$, and the remaining two neighbors must be either vertical or to the left of vertical, meaning their horizontal components have negative sum. Thus, there would have to be two neighbors between $Q_1$ and $Q_2$. But then if there was an additional neighbor above $PQ$ but not below, say, then the vertical components would have positive sum and so would not cancel. This shows that it is impossible for $Q$ to have 5 neighbors at distance $\sqrt{3}.$ 

Thus, in order for $Q$ to be balanced at distance $\sqrt{3}$, it must have either 2, 4, or 6 neighbors at that distance, and by analogous reasoning, $P$ must as well. Suppose now that they both have exactly 6 neighbors at this distance, as shown in Figure \ref{fig:33_1}. 

\begin{figure}[H]
    \centering
    \includegraphics[width=7cm]{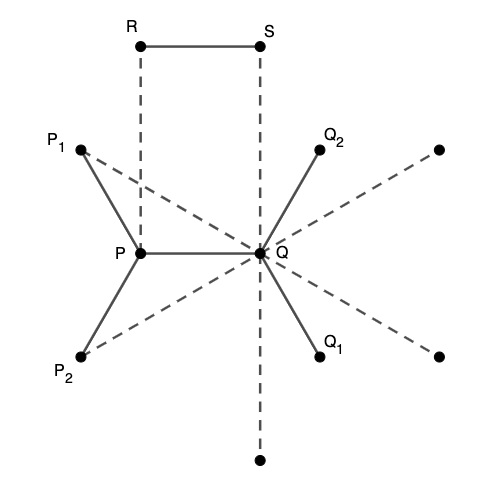}
    \caption{Neighbors of $Q$ at distance $\sqrt{3}$ in the case $(3,3)$}
    \label{fig:33_1}
\end{figure}

Let $R$ be the neighbor of $P$ above $PQ$ between $P_1$ and $Q_2$, and $S$ the neighbor of $Q$ above $PQ$ between these points. Then since $\angle RPQ, \angle SQP \le 90^\circ$, the only way $RS\ge 1$ is if $RP$ and $SQ$ are both perpendicular to $PQ$, in which case $RS=1$. By analogous reasoning, the neighbor of $Q$ below line $PQ$ between $P_2$ and $Q_1$ must line on line $QS$, and then the only way for the remaining points at distance $\sqrt{3}$ from $Q$ to balance out is if all 6 neighbors are evenly spaced in $60^\circ$ angles.

Now we know that either $P$ and $Q$ both have 6 evenly spaced neighbors at distance $\sqrt{3}$, or at least one (say $Q$) of them has 4 neighbors at distance $\sqrt{3}$. In either case, the reflections $P_1'$ of $P_1$ about $Q$ and $P_2'$ of $P_2$ about $Q$ must be in our set (Figure \ref{fig:33_2}). 

But then $P_2'$ is a neighbor of $Q_2$ (at distance 1) and $\angle QQ_2P_2'=120^\circ$, so $Q_1$ must have exactly 3 neighbors, since if it had 4, 5, or 6, then one of the other neighbors would be at a $60^\circ$ angle to $Q$ and so would also be a neighbor of $Q$. Similarly, $Q_1$ must have exactly 3 neighbors. It follows that $Q$ actually has exactly 6 neighbors at distance $\sqrt{3}$. From our earlier argument, this implies that either $P$ has 4 neighbors or 6 evenly spaced neighbors at distance $\sqrt{3}$, so the reflections of $Q_2$ and $Q_1$ about $P$ are also in our set, which in turn implies that $P_1$ and $P_2$ each have exactly 3 neighbors.

\begin{figure}[H]
    \centering
    \includegraphics[width=7cm]{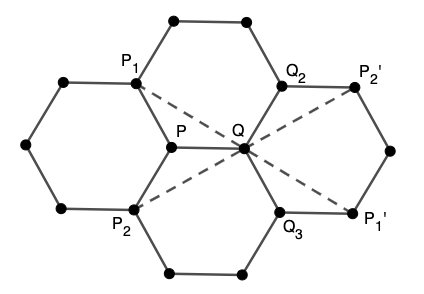}
    \caption{Regular hexagon tiling for the case $(3,3)$}
    \label{fig:33_2}
\end{figure}

We can continue in this manner to show that all points in our configuration have exactly 3 nearest neighbors, which forces to points to be the vertices of a regular hexagon tiling as shown in Figure \ref{fig:33_2}. There is no space for additional points inside one of these hexagons without being within distance 1 of some vertex, so this is the only configuration possible for this case.

\subsection{Case \texorpdfstring{$(3,2)$}{(3,2)}}

The points $P_1$ and $P_2$ must both be at a distance of $\sqrt{3}$ from $Q$, and they make a $60^\circ$ angle at $Q$. We claim that there must be either exactly 4 or exactly 6 points at this distance.
\begin{figure}[H]\label{fig:32no5}
    \centering
    \includegraphics[width=6cm]{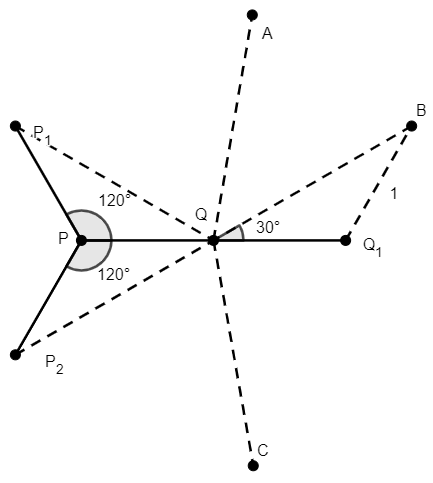}
    \caption{Case $(3,2)$ with 5 points at distance $\sqrt{3}$ from $Q$}
\end{figure}

There must be at least two other points at distance $\sqrt{3}$ from $Q$. If there are 3 other points like in Figure~\ref{fig:32no5}, then there must be either two above or two below the horizontal line $PQ_1$; suppose without loss of generality that there are two above, and we'll refer to the 3 points as $A,B,$ and $C$ as in the figure. Since $Q_1B \geq 1$, we know $\angle BQQ_1 \geq 30^\circ$. Then the vertical component of the vector $\vec{QB} + \vec{QA}$ is more than $\sqrt{3}$ because $A$ must be even higher than $B$. Then since the vectors $\vec{QP_1}$ and $\vec{QP_2}$ cancel vertically, we need the vertical component of $\vec{QC}$ to be even less than $-\sqrt{3}$, which is impossible since $QC = \sqrt{3}$. Therefore $Q$ does not have exactly 5 neighbors at distance $\sqrt{3}$.

Next, we claim that $Q$ cannot have 7 or 8 neighbors at distance $\sqrt{3}$. If so, there must be either 4 neighbors above line $PQ$ or 4 below, so assume there are 4 above, and that these points are $P_1, R, S,$ and $T$, in clockwise order. Then if we consider the balance about $R$ at distance $P_1R$, there is no space for $R$ to have more than 2 neighbors at this distance, because the next neighbor counterclockwise from $P_1$ would be too close to at least one of $P_1,P,$ $Q$, or $S$. Thus, the only possibility is that the reflection $P_1'$ of $P_1$ about $R$ is in our set. Similarly $S$ can only have one other neighbor at distance $ST$, so the reflection $T'$ of $T$ about $S$ is in our set. Then we must have $P_1'=T'$, since otherwise $P_1'$ and $T'$ would be too close to each other. But then there is only one possible location of $R, S,$ and $T$ satisfying these constraints, namely, the configuration shown in Figure \ref{fig:32_2} with $P_1R=RP_1'=P_1'S = ST\approx 1.02$. 
\begin{figure}[H]
    \centering
    \includegraphics[width=6cm]{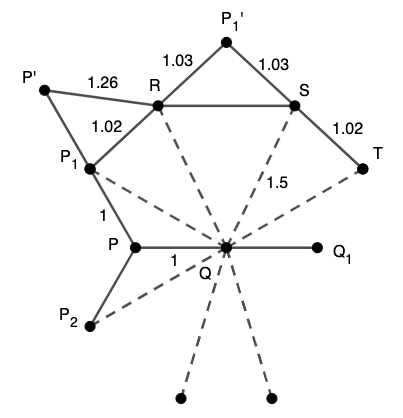}
    \caption{Case $(3,2)$ with 7 points at distance $\sqrt{3}$ from $Q$}
    \label{fig:32_2}
\end{figure}
Since we've already considered the cases $(5,3),$ $(4,3)$, and $(3,3)$, the only possibility is that $P_1$ has 2 neighbors at the minimal distance, so the reflection $P'$ of $P$ about $P_1$ is in our set. Then we get $P'R\ne RS$ (since $P'R\approx 1.26$ while $RS\approx 1.50$), but there is no space for another neighbor of $R$ at either of these distances except if it's above $P'R$ and $RP_1'$ without getting too close to one of $P_1, P, Q,$ or $S$. This makes it impossible for $R$ to be balanced at either of these distances, giving a contradiction. It follow that $Q$ cannot have 7 or more neighbors at distance $\sqrt{3}$.

The only remaining possibility is that $Q$ has 4 or 6 neighbors at distance $\sqrt{3}$. Now since we've already considered the cases $(5,3), (4,3)$, and $(3,3)$, we know that each of $P_1,P_2,P_1',$ and $P_2'$ can only have one other neighbor. Thus, by the same argument, $P_1$ also has either 4 or 6 neighbors at distance $\sqrt{3}$.

If both of them have 6 neighbors, then the only possibility is that both of them have neighbors evenly spaced at $60^\circ$ angles, since otherwise two of their neighbors would be too close together: Figure \ref{fig:32_3} shows the extreme case where the next neighbor $R$ of $P_1$ at distance $\sqrt{3}$ counterclockwise from $Q$ is as far as possible from the next neighbor $S$ of $Q$ at distance $\sqrt{3}$ clockwise from $P_1$, and we see in this case that $RS\approx 0.55<1$. We can compute this extreme case by making the angle $QP_1R$ as large as possible subject to the conditions that the components in the direction of $\vec{PP_1}$ cancel and that $\angle RP_1T$ is large enough to ensure that $RT>1.$

But then this common neighbor $R$ must also be a neighbor at distance $\sqrt{3}$ to the reflection $P_2'$ of $P_2$ about $Q$, and its reflection $R'$ about $Q$ must be a neighbor at distance $\sqrt{3}$ to point $P_2$ and to the reflection $P_1'$ of $P_1$ about $Q$, as shown in the second half of Figure \ref{fig:32_3}, so each of these points must also have 6 neighbors at distance $\sqrt{3}$. It then follows by symmetry that the neighbor (at distance 1) on the other side of each of these points must itself have 3 neighbors (at distance 1), and then each of that point's other neighbors must have 2 neighbors (at distance 1) and 6 neighbors at distance $\sqrt{3},$ and so on.

\begin{figure}[H]
    \centering
    \includegraphics[width=13cm]{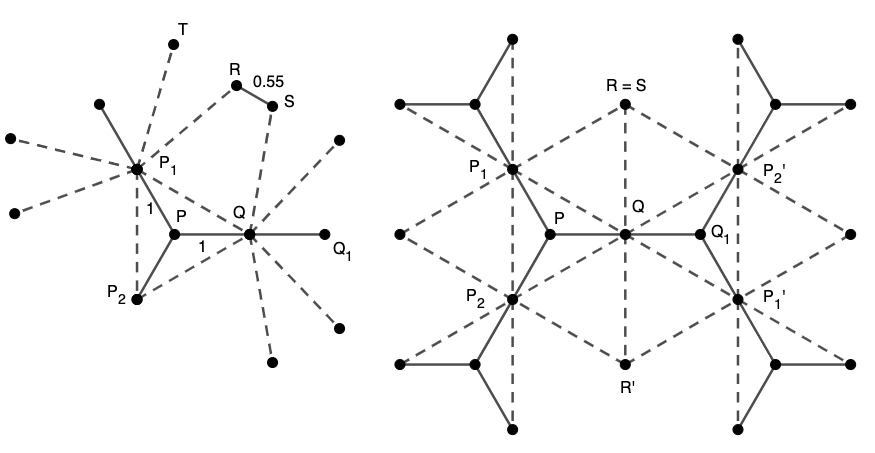}
    \caption{Case $(3,2)$ where both $P_1$ and $Q$ have 6 neighbors at distance $\sqrt{3}$}
    \label{fig:32_3}
\end{figure}

We can continue this process to get the second configuration shown in Figure \ref{fig:32_4}, which consists of a regular hexagon tiling where the hexagons have side length 2 and the points in our configuration are the vertices, midpoints of the edges, and centers of the hexagons. No other points can be added to this configuration without being within distance 1 of one of these points, so this is the only possible configuration in the case where both $Q$ and $P_1$ have 6 neighbors at distance $\sqrt{3}$.
\begin{figure}[H]
    \centering
    \includegraphics[width=16
    cm]{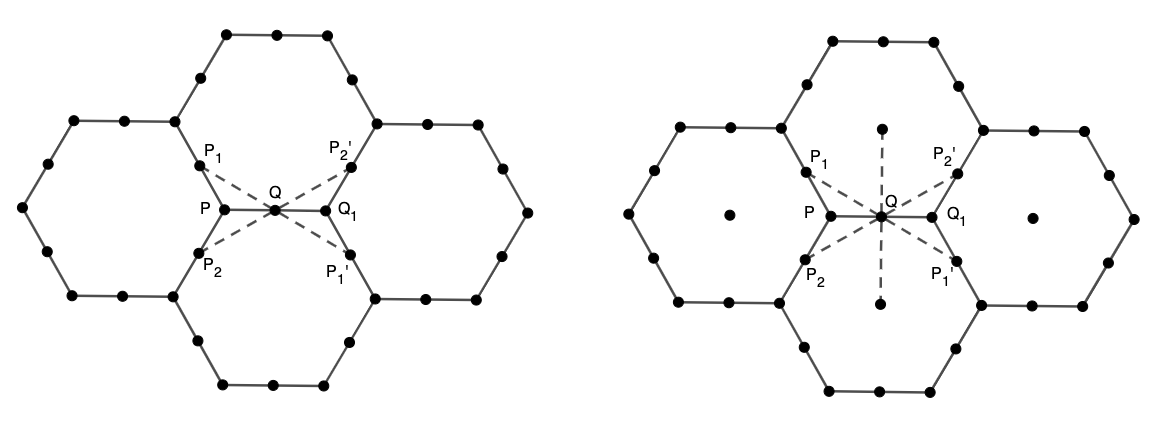}
    \caption{Regular hexagon tiling for the case $(3,2)$, with or without the centers of the hexagons}
    \label{fig:32_4}
\end{figure}
It remains to consider the case where at least one of $Q$ or $P_1$ has only 4 neighbors at distance $\sqrt{3}$. If $Q$ has 4 such neighbors, then the reflections $P_1'$ and $P_2'$ of $P_1$ and $P_2$ about $Q$ must be in $\mc{C}$. Then by symmetry, $P_1'$ and $P_2'$ are neighbors of $Q_1$ at distance 1, which implies that $Q_1$ has 3 such neighbor. This in turn implies that each of $P_1'$ and $P_2'$ has just one other neighbor.

If $P_1$ has 6 neighbors at distance $\sqrt{3}$, then let $R$ and $T$ be the next two such neighbors counterclockwise from $Q$. Now if $P_2'$ had 6 neighbors at distance $\sqrt{3}$, one of them would be too close to $R$, as shown in Figure \ref{fig:32_5}. Thus, $P_2'$ has only 4 neighbors at distance $\sqrt{3}$, so by the same argument as before, this implies that its other neighbor $R_1$ has exactly 3 neighbors at distance 2. Let $S_1$ be its next neighbor clockwise from $P_2'$. Then $S_1$ also cannot have 6 neighbors at distance $\sqrt{3}$, or else one of them would be too close to $R$, so it must have 4 such neighbors, and then the reflection $R_2$ of $S_1$ about $R$ must also have 3 neighbors for the same reason as before.

Continuing in this manner, we can see that there must be a regular hexagon $PQ_1R_1R_2R_3R_4$ as shown in Figure \ref{fig:32_5}, with all its vertices in $\mc{C}$ as well as the midpoints $P_1,Q,P_2',S_1,S_2,S_4$ of its edges. But then point $T$ will be too close to $S_3$, which forces $T=S_3,$ implying that all the neighbors of $P_1$ at distance $\sqrt{3}$ are actually evenly spaced at $60^\circ$ angles. But then $R$ would also be a neighbor of $Q$ at distance $\sqrt{3}$, which is a contradiction.

This shows that if either of $P_1$ or $Q$ has only 4 neighbors at distance $\sqrt{3}$, then actually both of them do. This implies that for every 2-neighbor point in our set, each of its neighbors is a 3-neighbor point, and since we know that every neighbor of a 3-neighbor point in the configuration is a 2-neighbor point, the configuration must contain all the vertices and edge midpoints of a regular hexagon tiling, as shown on the left in Figure \ref{fig:32_4}.

\begin{figure}
    \centering
    \includegraphics[width=12cm]{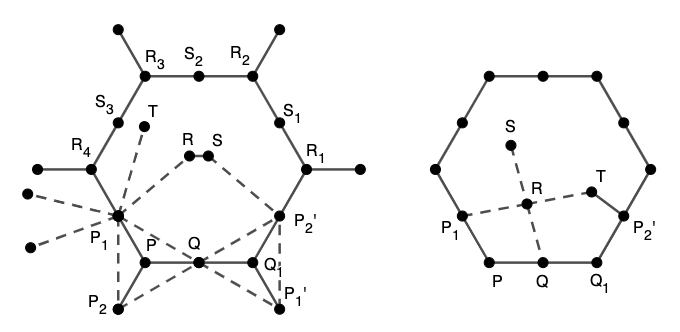}
    \caption{Showing the existence of hexagons and lack of other points inside them in the case $(3,2)$}
    \label{fig:32_5}
\end{figure}

It remains to show that there is no space for additional points inside the hexagons (unless we add the center of one hexagon, in which case we would be forced to add all the centers as explained above). To see this, suppose for contradiction there is another point $R$ inside the hexagon, and without loss of generality let $P_1$ and $Q$ be the two closest edge-midpoints of the hexagon to $R$. We must have $\angle P_1RQ>120^\circ,$ or else $RP<1.$ Thus, $R$ cannot have 3 neighbors inside the hexagon at distance $RP_1$ or $RQ$. There also is not space for it to have 5 or more neighbors at either distance without one of those neighbors getting too close to $P$ or to the other of $P_1$ or $Q$.

This implies that the reflections $S$ of $Q$ about $R$ and $T$ of $P_1$ about $R$ must be in $\mc{C}$. But if we assume $\angle RPQ \le \angle RPP_1,$ then $T$ will be too close to the edge-midpoint $P_2'$, because if $PP_1RQ$ were a rhombus, then we would have $TP_2'$, and making $QR$ larger than 1 and $\angle RPQ$ smaller than $60^\circ$ will only shrink $TP_2'$. This gives a contradiction, so we conclude that there is no space for more points inside the hexagons, and that the only possibilities for the case $(3,2)$ are the two hexagon tilings shown in Figure \ref{fig:32_4}.

\subsection{Case \texorpdfstring{$(2,2)$}{(2,2)}}

If $Q_1$ or $P_1$ had more than 2 neighbors, that would create an instance of one of the cases $(3,2), (4,2), (5,2),$ or $(6,2)$, and in each of those cases we determined that there were never two adjacent vertices that each had two neighbors. Therefore $Q_1$ and $P_1$ also have two neighbors, which must also be on the line $PQ$, and then repeating this reasoning tells us that we must have points all along the line $PQ$, with a distance of 1 between consecutive points. One possibility is that this is the entire configuration, but there could be more points.

Let $L$ be the line $PQ$ and $\mc C$ be our balanced configuration, which we will assume is not contained in $L$. Then let $d^* = \inf \{d(R,L) : R\in \mc C\setminus L\}$, where $d(R,L)$ is the distance from the point $R$ to any point on the line $L$. For any $\eps > 0$, there exists a point $R \in \mc C\setminus L$ such that $d(R,L) < d^*+\eps$. The line through $R$ perpendicular to $L$ intersects $L$ between two points in $\mc C$ at distance 1, which we will assume are $P$ and $Q$ without loss of generality. At least one of $\angle RPQ$ and $\angle RQP$ are greater than $60^\circ$; we suppose that is $\angle RPQ$. Now, to balance at $P$ with distance $PR$, we must have one or more other points at this distance. If there are two or four points at distance $PR$ from $P$, then the point $R'$, which is opposite from $R$ with respect to $P$, must be in the configuration. If there are three, then those points are at $120^\circ$ angles and one must be within $d^*$ of $L$, as long as $\eps < d^*/10$ for example. Otherwise, if there are five or more points on the circle of radius $PR$ centered at $P$, then at least three must be in the arc which is at least $d^*$ away from $L$, as depicted in Figure~\ref{fig:22_1}. The points at opposite ends of this arc are separated by a distance of less than $1+2\eps$, since $R$ is closer to $P$ than $Q$. For, say, $\eps < 1/10$ this is impossible without two of the points on the arc being less than 1 apart from each other, so in fact we must have $R' \in \mc C$.

\begin{figure}[H]
    \centering
    \includegraphics[width=16cm]{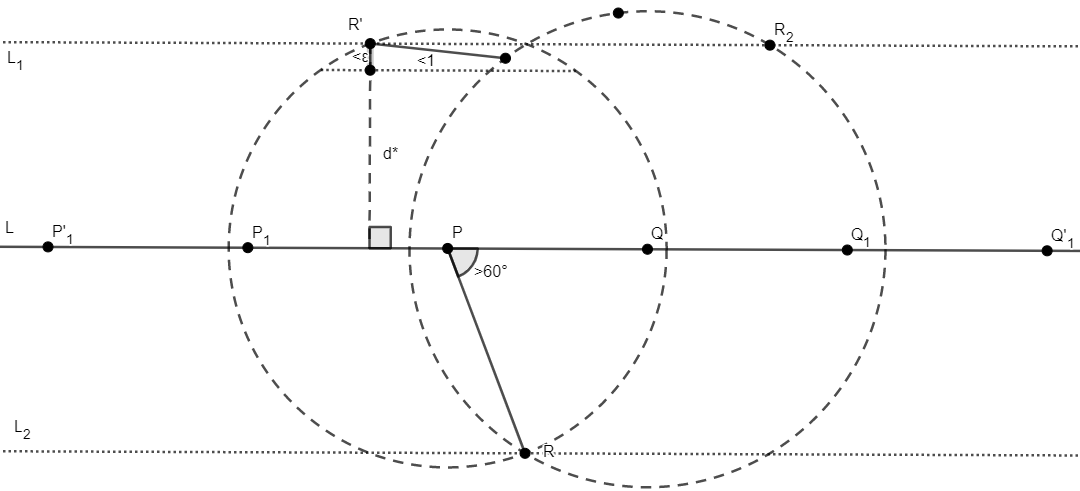}
    \caption{Case $(2,2)$, demonstrating that the points $R'=R_1$ and $R_2$ must be in the configuration.}
    \label{fig:22_1}
\end{figure}

Now, we can try to apply the same argument to the reflection of $R$ about $Q$, but then the points at opposite ends of the arc of points on the circle of radius $RQ$ about $Q$ which are at least $d^*$ away from $L$ is only bounded above by $2\eps + 2$. Then there may be room for three points at least 1 apart from each other on this arc. The leftmost of these points would then be less than 1 away from $R'$, which is impossible. 
Therefore the reflections of $R$ about $P$ and $Q$ both exist in $\mc C$. Call these points $R_1$ and $R_2$, respectively.

The points $R_1$ and $R_2$ are separated by a distance of 2 on the line $L_1$ which is parallel to $L$ at a distance less than $d^* + \eps$. Then we can repeat the same argument with $R_1$, which is between $P$ and $P_1$, to get a point $R_0$ which is the reflection of $R_1$ about $P_1$. Similarly we reflect $R_2$ about $Q_2$ to get $R_3$. Then $R_0,R,$ and $R_3$ are all on another line $L_2$ parallel to $L$, and are also separated by a distance of 2. Repeating this reasoning tells us that the lines $L_1$ and $L_2$ must both have points separated by distances of 2 going off in both directions.

The lines $L_1$ and $L_2$ are each at a distance at most $d^* + \eps$ and at least $d^*$ from $L$. If there are no points at distance exactly $d^*$ from $L$, then there must exist a point $X$ at distance at most $d^* + \eps/3$ from $L$ and another point $Y$ at a distinct distance at most $d^* + \eps/9$ away from $L$. By the same reasoning, there must be points in the configuration all along the lines through $X$ and $Y$ parallel to $L$, separated by a distance of 2. Then we can draw a $2.1$-by-$\eps$ box containing four points in the configuration, which is impossible without two of the points being at a distance less than 1 from each other. Therefore there is a point at distance exactly $d^*$ from $L$, and we can assume these points are on the lines $L_1$ and $L_2$, so that there are no points in the strip between $L_1$ and $L$ and between $L$ and $L_2$. Moreover, there cannot be any points less than $2d^*$ away from $L$, other than those on $L,L_1,L_2$, because a point less than $d^*$ away from $L_1$ (for example) would necessitate a point in the strip between $L_1$ and $L$ when we balance about a point in $L_1$.

Now, we return to the point $R$ on $L_2$ which is $d^*$ away from $L$. The only point within $2d^*$ of $L$ which is at a distance $PR$ from $R$ is $P$ itself, so there must be another point more than $2d^*$ away from $L$ which is $PR$ from $R$. There can only be one such point, otherwise the average of the points at $PR$ from $R$ would be further from $L$ than $R$. Therefore the reflection of $P$ about $R$ must be in the configuration, since otherwise we would not be balanced at $R$ with distance $PR$. The same argument shows that the reflection of $Q$ about $R$ is in the configuration, and similarly for all the other points on $L_2$. These reflections create another line $L_4$, where there are points in the configuration spaced at distance 1 from each other. There is an analogous line $L_3$ on the opposite side of $L$. The lines $L_3$ and $L_4$ are at a distance of exactly $2d^*$ from $L$, so just like before, we know that all points at a distance of less than $3d^*$ much be on one of the lines $L,L_1,L_2,L_3,L_4$, since a closer point would force a new point less than $2d^*$ away from $L$, by balancing about a point in $L_3$ or $L_4$.

The lines $L_3$ and $L_4$ are in the same situation as the line $L$ was to begin with; they contain points in the configuration separated by a distance of 1 all along them, and the nearest point is at a distance $d^*$. Then the entire string of reasoning applies repeatedly, to produce infinitely many lines $L_i$ separated at distance $d^*$ from each other, with points spaced at distance 1 along them when $i$ is congruent to $0$ or $3$ modulo 4 and spaced at distance 2 when $i$ is congruent to $1$ or $2$ modulo 4. These lines cover the entire plane, leaving room for additional points in the configuration only halfway between consecutive points on the lines with points spaced at distance 2. If one of these halfway points exists, then all of them in the entire plane must, as this point can be repeatedly reflected about the nearest points on its line and the neighboring lines to get to any other halfway point.

We are left with three possibilities. The configuration $\mc C$ could be entirely contained in the line $L$, which is balanced when the points are evenly spaced on $L$. If there is any point not on $L$, then at a minimum there must be infinitely many lines of points evenly spaced at distances 1 and 2, alternately. The points on lines with points spaced at distance 1 form the points of a lattice, and the points on the lines spaced at distance 2 are the midpoints of the edges of that lattice. Any lattice together with its edge midpoints comes up in this form, and it is a balanced configuration. Finally, we can add in the points halfway between consecutive points spaced at distance 2, which is like adding in the centers of the lattice tiling as well, or alternatively like taking the points of the same lattice tiling with generators half as large. This configuration is also balanced, and we have shown that these are the only possible balanced configurations in the case (2,2).

\begin{figure}[H]
    \centering
    \includegraphics[width=8cm]{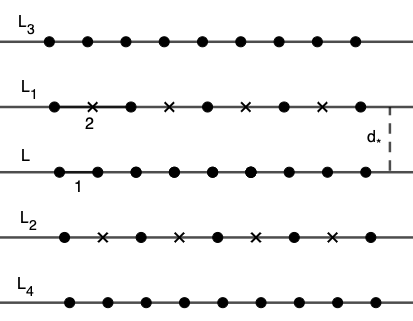}
    \caption{The possible balanced configurations not contained in $L$, with optional halfway points marked.}
    \label{fig:22_result}
\end{figure}

\newpage

\section{Balanced configurations in hyperbolic space}\label{sec:hyperbolic}

The notions of balanced and group-balanced configurations in spherical and Euclidean space can naturally be extended to $n$-dimensional hyperbolic space $\mb{H}^n$ as follows, using a definition analogous to the one suggested in Remark \ref{rem:alternate_def} for the spherical case:

\begin{definition}
Given $\mc{C}\se \mb{H}^n$, take a conformal embedding of $\mb{H}^n$ in some higher dimensional Euclidean space, and for each $P\in \mc{C}$, consider the $n$-dimensional Euclidean hyperplane tangent to $\mb{H}$ at $P$. Then for each distance $d\ge 0$, let $P_1,\dots,P_m\in \mc{C}$ be the neighbors of $P$ at distance $d$, and let $\vec{PP_1'},\dots,\vec{PP_m}'$ be equal length vectors in this hyperplane such that $\vec{PP_i'}$ is tangent to the hyperbolic geodesic $PP_i'$ for each $i$. Then $\mc{C}$ is \emph{balanced} if $\vec{PP_1'}+\dots+\vec{PP_m'}=\vec{0}$ for every $P\in\mc{C}$ and every $d\ge 0.$
\end{definition}

\begin{definition}
A configuration $\mc{C}\se\mb{H}^n$ is \emph{group-balanced} if for any $P\in\mc{C}$, there is a symmetry of $\mc{C}$ fixing no points of $\mb{H}^n$ except $P$.
\end{definition}

For the case $n=2$, one can construct two families of discrete group-balanced configurations in $\mb{H}^2$ which are analogous to those described in Theorems \ref{thm:leech_classification} and \ref{thm:classification} for the spherical and Euclidean cases.

\subsection{Configurations in \texorpdfstring{$\mb{H}^2$}{H2} analogous to regular polygon tilings}

Let $p,q,$ and $r$ be any positive integers with $1/p+1/q+1/r<1.$ Then up to isometry, there is a unique triangle in $\mb{H}^2$ with angles $180^\circ/p,180^\circ/q,$ and $180^\circ/r$, since the angle sum of that triangle is less than $180^\circ$. One can construct a tiling of $\mb{H}^2$ using triangles of this shape by repeatedly reflecting each triangle in the tiling across all three of its edges. Then there are three distinct types of vertices in this tiling:
    \begin{itemize}
        \item The points at which $2p$ of the vertices with angle $180^\circ/p$ meet, 
        \item The points at which $2q$ of the vertices with angle $180^\circ/q$ meet, 
        \item The points at which $2r$ of the vertices with angle $180^\circ/r$ meet.
    \end{itemize}
One can construct a group-balanced configuration in $\mb{H}^2$ by taking any one, two, or all three of these sets, since then there will be either a $p$-fold, $q$-fold, or $r$-fold rotational symmetry about that point. The symmetry group fixing each point is one of the three dihedral groups $D_p$, $D_q$, or $D_r$, depending which of the three categories the point is in, and the symmetry group of the full tiling is the Coxeter group whose Coxeter-Dynkin diagram is a triangle with edges labeled $p,q,$ and $r$ and vertices corresponding to the reflections across the three different types of edges in the tiling.

\subsubsection*{Corresponding configurations in the spherical case}

In the spherical case, we can get an analogous set of tilings for any triple $(p,q,r)$ with $p,q,r\ge 2$ and $1/p+1/q+1/r>1$, since any spherical triangle must have angle sum more than $180^\circ$. For the triples $(2,2,n)$, we get the configurations coming from $n$ equally spaced point along some great circle, with the 2's corresponding to these vertices and the edge midpoints, and the $n$ to the two antipodal points. For $(2,3,3)$ we get the configurations coming from a regular tetrahedron (where the 2 corresponds to the edge midpoints, the 3 to the vertices, and the 3 to the face centers). For $(2,3,4)$, we get the configurations coming from a regular octahedron (or equivalently, its dual cube), with the 2 corresponding to the edge midpoints, the 3 to the cube vertices, and the 4 to the octahedron vertices. Similarly, for $(2,3,5)$ we get the balanced configurations coming from a regular icosahedron (or its dual dodecahedron), with the 2 corresponding to the edge midpoints, the 3 to the dodecahedron vertices, and the 5 to the icosahedron vertices. Thus, all balanced configurations on the sphere fall under this category.

\subsubsection*{Corresponding configurations in the Euclidean case}

In the Euclidean case, we get tilings corresponding to the triples $(p,q,r)$ with $1/p+1/q+1/r$, since any Euclidean triangle must have angle sum exactly $180^\circ.$ For the case $(2,3,6)$, we recover the regular hexagon tilings from Theorem \ref{thm:classification}: each hexagon is made up of 12 of these tiles, with the 2 corresponding to the edge midpoints, the 3 to the vertices, and the 6 to the face centers. We also get the cases $(2,4,4)$ and $(3,3,3)$, giving the tilings of the plane with squares and equilateral triangles, but these do not need to be listed as separate cases in Theorem \ref{thm:classification} because the resulting configurations already fall under the lattice parallelogram case.
    
\subsection{Configurations in \texorpdfstring{$\mb{H}^2$}{H2} analogous to lattices}

Consider a triangle with arbitrary angles $\alpha,\beta,$ and $\gamma$ summing to $360^\circ/m$ for some positive integer $m\ge 3.$ One can tile $\mb{H}^2$ with triangles of this shape by repeatedly rotating each triangle in the tiling about the midpoint of each of its edges. All vertices of this tiling look the same, with a total of $3m$ angles about each point in the pattern $\alpha,\beta,\gamma,\alpha,\beta,\gamma,\dots,$ so that there is an $m$-fold rotational symmetry about each vertex. Then we can consider the following four sets of points:
\begin{itemize}
    \item The vertices of the tiling,
    \item The midpoints of the edges between angles of sizes $\alpha$ and $\beta,$
    \item The midpoints of the edges between angles of sizes $\alpha$ and $\gamma,$
    \item The midpoints of the edges between angles of sizes $\beta$ and $\gamma.$
\end{itemize}
One can construct a group-balanced configuration in $\mb{H}^2$ by taking any one, two, three, or all four of these sets, since there is an $m$-fold rotational symmetry about each vertex and a 2-fold rotational symmetry about each edge midpoint.

\subsubsection*{Corresponding configurations in the Euclidean case}
    
These are analogous to lattice parallelograms tilings in the planar case. We require $m\ge 3$ in the hyperbolic case to have angle sum less than $180^\circ$, but in the planar case we would have $m=2$, and this leads to the lattice parallelogram tilings, where we take pairs of adjacent triangles in the tiling to form the parallelograms. Then the vertices of the triangles correspond to the vertices of the parallelograms, two of the types of edge midpoints correspond to two of the types of edges of the parallelogram tiles, and the third type of edge midpoints (namely, the midpoints of the sides along which we join the triangles to form the parallelograms) corresponds to the centers of the faces of the parallelograms. Note that given a scalene triangle and a lattice spanned by two of its edge vectors, we could construct the parallelograms in three different ways (by joining the triangles in the tiling along any of the three types of edges) but we would get the same balanced configurations in each case since we are just changing which points are considered edge midpoints and which are considered face centers.

\subsection{Open questions for the hyperbolic case}

We can pose the following questions, similar to those posed in \cite{cohn2010} for the Euclidean case:

\begin{question}
Is this the full set of discrete group-balanced configurations in $\mb{H}^2$, or are there others?
\end{question}

\begin{question}
Are all discrete balanced configurations in $\mb{H}^2$ group-balanced?
\end{question}

\begin{question}
Is there some positive integer $n$ such that there is a discrete configuration in $\mb{H}^n$ which is balanced but not group-balanced?
\end{question}

It is possible the $n=2$ case could be resolved using a case analysis similar to ours and Leech's \cite{Leech}, but it would likely require a much more involved argument, as there would be infinitely many cases to consider.

\section{Acknowledgements}

We thank Henry Cohn for the problem suggestion and many helpful conversations related to this work.

\bibliographystyle{plain}
\bibliography{references}

\end{document}